\documentclass{amsart}
\usepackage{newlfont,amsfonts,amssymb,% oldgerm,
amsmath,amsthm,amsgen,amscd,wasysym,comment}
\renewcommand{\d}{{\mathrm d}}
\newcommand{\bcase}{\begin{case}}
\newcommand{\ecase}{\end{case}}

\newcommand{\bclaim}{\begin{claim}}
\newcommand{\eclaim}{\end{claim}}

\newcommand{\bstep}{\begin{step}}
\newcommand{\estep}{\end{step}}

\newcommand{\bhlem}{\begin{hlem}}
\newcommand{\ehlem}{\end{hlem}}

\newcommand{\bleer}{\begin{leer}}
\newcommand{\eleer}{\end{leer}}
\newcommand{\bde}{\begin{de}}
\newcommand{\ede}{\end{de}}

\newcommand{\ol}{\overline}

\newcommand{\bs}{\begin{satz}}
\newcommand{\es}{\end{satz}}
\newcommand{\btheo}{\begin{theo}}
\newcommand{\etheo}{\end{theo}}
\newcommand{\bfolg}{\begin{folg}}
\newcommand{\efolg}{\end{folg}}
\newcommand{\blem}{\begin{lem}}
\newcommand{\elem}{\end{lem}}
\newcommand{\bnote}{\begin{note}}
\newcommand{\enote}{\end{note}}
\newcommand{\bprf}{\begin{proof}}
\newcommand{\eprf}{\end{proof}}
\newcommand{\bd}{\begin{displaymath}}
\newcommand{\ed}{\end{displaymath}}
\newcommand{\be}{\begin{eqnarray*}}
\newcommand{\ee}{\end{eqnarray*}}
\newcommand{\eeqa}{\end{eqnarray}}
\newcommand{\beqa}{\begin{eqnarray}}
\newcommand{\bi}{\begin{itemize}}
\newcommand{\ei}{\end{itemize}}
\newcommand{\bnum}{\begin{enumerate}}
\newcommand{\enum}{\end{enumerate}}

\newcommand{\ve}{\varepsilon}

\newcommand{\beq}{\begin{equation}}
\newcommand{\eeq}{\end{equation}}
\newcommand{\einhalb}{\frac{1}{2}}
\newcommand{\rr}{\mathbb{R}}

\newcommand{\vf}{\varphi}
\newcommand{\earr}{\end{array}\]}
\newcommand{\barr}{\[\begin{array}}
\newcommand{\bvec}{\left(\begin{array}{c}}
\newcommand{\evec}{\end{array}\right)}

\newcommand{\w}{\omega}

\newcommand{\del}{\partial}

\newcommand{\bbem}{\begin{bem}}
\newcommand{\ebem}{\end{bem}}
\newcommand{\bbez}{\begin{bez}}
\newcommand{\ebez}{\end{bez}}
\newcommand{\bbsp}{\begin{bsp}}
\newcommand{\ebsp}{\end{bsp}}

\newcommand{\D}{\Delta}

\newcommand{\wt}{\widetilde}

\newcommand{\tem}{\widetilde{M}}
\newcommand{\tg}{\widetilde{g}}

\newcommand{\ro}{\mathsf{P}}

\newcommand{\cur}{{\cal R}} 

 \newcommand{\hook}{\raisebox{-0.35ex}{\makebox[0.6em][r]
{\scriptsize $-$}}\hspace{-0.15em}\raisebox{0.25ex}{\makebox[0.4em][l]{\tiny
 $|$}}}
\newcommand{\der}{{\mathrm d}}
\newcommand{\ups}{\Upsilon}
\newcommand{\dz}{\wedge}
\newcommand{\bbC}{\mathbb{C}}
\newcommand{\bbR}{\mathbb{R}}

\theoremstyle{definition}
\newtheorem{de}{Definition}
\newtheorem{bem}{Remark}
\newtheorem{bez}{Notation}
\newtheorem{bsp}{Example}
\theoremstyle{plain}
\newtheorem{lem}{Lemma}
\newtheorem{satz}{Proposition}
\newtheorem{folg}{Corollary}
\newtheorem{theo}{Theorem}

\newtheorem{theorem}{Theorem}
\newtheorem{lemma}{Lemma}

\newtheorem{proposition}{Proposition}
\newtheorem{corollary}{Corollary}

\theoremstyle{definition}
\newtheorem{definition}{Definition}

\theoremstyle{definition}

\begin{document}
\bibliographystyle{abbrv}
\title{Ambient metrics for $n$-dimensional $pp$-waves} 
\author{Thomas Leistner}\address{Department Mathematik, Universit\"at Hamburg, Bundesstra\ss e 55, D-20146 Hamburg, Germany} \email{leistner@math.uni-hamburg.de}
\author{Pawe\l~ Nurowski} \address{Instytut Fizyki Teoretycznej,
Uniwersytet Warszawski, ul. Ho\.za 69, 00-681 Warszawa, Poland}\email{nurowski@fuw.edu.pl}\thanks{This
  work was supported in part by the Polish Ministerstwo Nauki i
  Informatyzacji grant nr: 1 P03B 07529 and by the Sonderforschungsbereich 676 of the German Research Foundation}

\date{\today}
% \begin{abstract}
% For a $pp$-wave of arbitrary dimension we determine explicitly the Fefferman-Graham ambient metric in those cases when it exists.
%\end{abstract}
\begin{abstract}
We provide an explicit formula for the {\sc Fefferman-Graham}-ambient metric of an $n$-dimensional 
conformal $pp$-wave in those cases where it exists. In even dimensions
we calculate the obstruction explicitly. 
Furthermore, we describe all $4$-dimensional $pp$-waves that are Bach-flat, and give a large class of  Bach-flat examples which are conformally Cotton-flat, but not conformally Einstein.
Finally, as an application, we use the obtained ambient metric to show that 
even-dimensional  $pp$-waves have vanishing
critical $Q$-curvature. 
\\[.3cm]
\end{abstract}
\maketitle

\section{Introduction}
Plane fronted gravitational waves, called \emph{$pp$-waves}, are 
Lorentzian 4-manifolds  $(M,g)$ admitting a 
\emph{covariantly constant null} vector field $K$. 
In addition, their Ricci tensor $Ric$ satisfies 
\begin{equation}
Ric= \Phi~ \kappa \otimes \kappa,
\end{equation}
where $\kappa$ is the 1-form on $M$ defined by $\kappa:= K \hook
g$. Physicists require also that 
the function $\Phi$ is nonegative for a $pp$-wave. This is because  
$\Phi$, via the \emph{Einstein field equations}, is directly related
to the energy momentum tensor of its gravitational field. 

$pp$-waves are important in general relativity theory since they 
generalize the concept of a \emph{plane wave of classical
  electrodynamics} 
\cite{robinson59}, as well as because of the fact that 
every 4-dimensional spacetime has a \emph{special} $pp$-wave 
as a well defined limit \cite{penrose76}, the Penrose limit, as it is called. 

Higher dimensional generalizations of the 4-dimensional $pp$-waves
were studied in \cite{schimming74}, appeared in  Kalu\-za-Klein
theory \cite{crusciel-kowalski84,kowalski84}, and later in string
theory 
\cite{berenstein-maldacena-nastase02,fof-penrose,
gauntlett-hull02,cvetic-lu-pope02}. Their
property of possessing a covariantly constant null vector field $K$, 
implies that they have \emph{reduced Lorentzian holonomy} from the full orthogonal group $\mathrm{SO}(1,n-1)$ to the subgroup preserving the null  vector $K$. 
In fact, they can be characterised by having {\em Abelian} holonomy $\rr^{n-2}$ \cite{leistner01,leistner05c}. As such they admit many \emph{supersymmetries}, which is a desirable feature of any string theory. For example, the dimension of the space of parallel spinors on an $n$-dimensional $pp$-wave is at least
% $2^{\left[\frac{n-2}{2}\right]}$, which is 
 half of the dimension of the spinor module, \cite{leistner01}.

In local coordinates $(x^i,u,r)_{i=1,\dots,n-2}$ in $\rr^n$ the $n$-dimensional $pp$-wave metric can be written as
$$
g=\sum_{i=1}^{n-2}(\der x^i)^2~+~ 2\der u ~\big(\der r+h\der u\big).
$$
Here $h$ is an arbitrary smooth real function of the first $(n-1)$
coordinates, $h=h(x^i,u)$. The covariantly 
constant null vector field is $K=\partial_r$. Another property of this
metric is that  it has vanishing scalar curvature. Hence, if it is \emph{Einstein} then it is \emph{Ricci
  flat}. This happens if and only
if $\Delta h=\sum_{i=1}^{n-2}\frac{\partial^2h}{\partial (x^i)^2}=0$.

\emph{Conformal classes} of $pp$-wave metrics have remarkable properties. 
One
of them is described by their discoverer H. W. Brinkmann already in
1925. In his seminal paper  \cite{brinkmann25} Brinkmann not only 
studied what was later called {\em Brinkmann wave}, namely Lorentzian 
manifolds with parallel null vector field, but he also showed the 
following \cite[Theorems IV and VIII]{brinkmann25}:  
{\em A $4$-dimensional Einstein manifold $(M,g)$  admits a 
function $\Upsilon$ such that 
the conformally rescaled metric ${\mathrm e}^{2\Upsilon}g$
is again Einstein, but \emph{not} homothetic to $g$,
 if and only if $(M,g)$ is a Ricci-flat $pp$-wave (or its counterpart in neutral signature\footnote{Be aware 
 %\edz{Footnote added, check it!}
  that the coordinates in the relevant Section 4.2 of Brinkmann's paper \cite{brinkmann25}  have to be understood as complex and complex conjugate in order to obtain Lorentzian metrics. If they are considered as real coordinates the resulting metric has neutral signature.}).}  
 In this case, the rescaled metric is also Ricci-flat and the gradient of $\ups$ is a null vector.
 %Here $\D_g$ denotes the Laplacian of $g$. 
 This occurs because the Weyl tensor $W
$ of a $pp$-wave is \emph{null} and \emph{aligned}  with $K$,
i.e. $K\hook W=0$,
 %$C^\mu_{~\nu\rho\sigma}K^\sigma=0$,
  which makes these metrics not \emph{weakly generic} in the terminology of \cite{gover/nurowski04}.

In this paper we discuss another remarkable conformal property of 
$n$-dimensional $pp$-wave metrics, which is related to the 
\emph{ambient metric construction} of Fefferman and Graham 
\cite{fefferman/graham85,fefferman-graham07}. The ambient 
metric construction mimics the situation in the flat model 
of conformal geometry: Here the $n$-dimensional sphere equipped 
with the flat conformal structure can be viewed as the 
projectivisation  of the light-cone in $(n+2)$-dimensional Minkowski 
space. Letting the spheres wandering along the light cone recovers the 
metrics in the conformal class. 
For a conformal class $[g]$ in signature $(p,q)$ on an
$n=(p+q)$-dimensional manifold $M$ the \emph{ambient metric} is a 
metric $\tg$ of signature $(p+1,q+1)$ on the product of $M$ with 
two intervals, $\widetilde{M}:=(-\ve,\ve)\times M\times 
(1-\delta,1+\delta)$, $\ve>0$, $\delta>0$, 
that is \emph{compatible with the conformal structure} (for details see 
Definition \ref{ambientmet}) and, moreover, is \emph{Ricci flat}\footnote{Note that in some papers from the physics literature the term Fefferman-Graham metric  has a different meaning than ours. What physicists call Fefferman-Graham metric, e.g. in \cite{spindeletal00} or \cite{deharo-skenderis-solodukhin01}, is a related concept that Fefferman and Graham call the Poincar\'e-Einstein metric. How to obtain one from another is well known and we shall explain it in Section \ref{qcurv}.}. 
The Ricci-flat condition ensures that the  the ambient metric depends 
uniquely on the conformal structure and encodes all properties of 
the conformal class $[g]$ but has the downside that the ambient metric does not always exist. Starting with a formal power series
\begin{equation} 
\tg = 2 \left( t \der\rho + \rho \der t \right) \der t + t^2
\left( g + \sum_{k=1}^\infty \rho^k\mu_k \right)\label{afg}
\end{equation}
with $\rho\in (-\ve,\ve)$, $t\in (1-\delta,1+\delta)$ Fefferman and Graham showed that if $n$ is \emph{odd},
the Ricci-flatness of the ambient metric gives equations for 
$ \mu_1,\mu_2, \ldots $ that can be solved {\em in principle}, but 
the calculations have been carried out only for very special conformal 
classes, mainly those that are related to Einstein spaces 
\cite{leitner05, leistner05a,gover-leitner06}.   
If $n=2s$ is \emph{even}, there is a conformally invariant \emph{obstruction} to the existence of a Ricci-flat ambient metric,  called the 
\emph{Fefferman-Graham obstruction}.  This obstruction  is the 
 nonvanishing of the obstruction tensor $\mathcal O$, given by the term 
$\mu_s$. In $n=4$ this obstruction tensor is  
the \emph{Bach tensor} for $g$. In higher dimensions the \emph{leading term}
of $\mathcal O$  is $\triangle^{s}_g(g)$, but there are a lot of lower
order terms, which, again, are determined  {\em in principle}, but
whose calculation is very cumbersome.
 
One important feature of the ambient metric is that 
if the metric $g$ is \emph{real analytic} then its corresponding
ambient metric $\wt{g}$ (if it exists) is \emph{also real analytic}
\cite{fefferman/graham85,fefferman-graham07,kichenassamy04}. 
  Another feature
of the ambient metric is that if the conformal class of $g$ includes an Einstein metric $g_E$, then
the power series in the ambient metric $\tilde{g}_E$ \emph{truncates}
at $k=2$; in particular, for $n>3$, even the obstruction tensor
vanishes. In such case the metric is given as a \emph{second order 
polynomial} in each of the variables $t$ and $\rho$. 
However, if the metric $g$ is 
\emph{not conformally Einstein}, then, except for a few examples 
\cite{gover-leitner06,nurowski07}, no explicit formulae for 
$\mu_k$, $k>3$ are known. 

 In this context our main result is the following remarkable conformal 
property of $n$-dimensional $pp$-waves: for them \emph{all} the 
coefficients $\mu_k$ in the ambient metric, the obstruction tensor in even dimensions, and hence, the condition under which the ambient metric truncates at a given order  can be  
calculated \emph{explicitly}. In Section \ref{ppambient} we prove
 \begin{theorem}\label{theo1}
 Let $g=\sum_{i=1}^{n-2}(\der x^i)^2~+~ 2\der u ~\big(\der r+h\der u\big)$ be
 an $n$-dimensional $pp$-wave metric with a real analytic
 function $h=h(x^1,\ldots, x^{n-2},u)$. Then the Feffermann-Graham 
ambient metric for the conformal class $[g]$ exists if and only if $n$ is \emph{odd} and $h$ is
arbitrary, or if $n=2s$ is \emph{even} and $\D^s h=0$. In both cases 
the ambient metric is given by a formal power series
 $$
\wt{g}
= 2 \der \left(t \rho  \right) \der t
 + t^2 \left(g
+ \left( \sum_{k=1}^{\infty} \frac{\D^kh}{k ! p_k }\rho^k \right)\ 
\der u^2\right),$$
with $p_k:=\prod_{j=1}^{k}(2j-n)$ and $\D:=\sum_{i=1}^{n-2}\del_i^2$.
 In particular, if $n=2s$ is even, the obstruction tensor $\mathcal O$ is given by $\mathcal O =\D^s h~ \der u^2$.
\end{theorem}
Thus if $n=2s$ is even, the ambient metric $\wt{g}$ is a
\emph{polynomial} of order $s-1$ in the variable $\rho$. If $n$ is
\emph{odd}, since the metric
$g$ is real analytic, Fefferman-Graham result guarantees that the
\emph{above metric} $\wt{g}$ is also \emph{real analytic}. This in 
particular means that the power series $\sum_{k=1}^{\infty} \frac{\D^kh}{k ! p_k
}\rho^k$ converges to a real analytic function in variable $\rho$.

Theorem \ref{theo1} provides us with a variety of examples of
conformal structures with {\em explicit} ambient metrics and which, 
in general, are \emph{not} conformally Einstein. For example, 
every polynomial $h$ in the $x^i$'s of order lower than $k$, with 
coefficients being functions of $u$, represents a 
$pp$-wave with ambient metric truncated at order lower than
$k/2$.  
In Section \ref{bachflat} we construct more general examples then
those defined by $h$ being polynomials in the $x^i$s. In particular, 
in dimension 
\emph{four} we find \emph{all} Bach-flat $4$-dimensional $pp$-waves and we
prove that most of them are {\em not conformally Einstein}. They are
defined by quite general functions $h$ and have ambient metrics which
are linear in variable $\rho$. It is interesting to note that these
$pp$-waves, although Bach-flat and conformal to Cotton-flat, are not
conformally Einstein.

Theorem \ref{theo1} implies also another interesting feature of the
$pp$-waves: their obstruction tensor $\mathcal O$ 
(in \emph{even} dimensions) involves only the terms of the highest
possible order in the derivatives of their metric; since \emph{all}
the lower order terms that are usually present in the obstruction
tensor are \emph{vanishing}, the $pp$-waves are, in a sense, 
the closest cousins of the conformally Einstein metrics.

Using the explicit form of the ambient metric and the main result of \cite{graham-juhl07}, in Section \ref{qcurv} we show that for even-dimensional  $pp$-waves the critical $Q$-curvature vanishes. This result is in correspondence with the fact that for a $pp$-wave all scalar invariants constructed from the curvature tensor vanish (for the proof in arbitrary dimension see \cite{CMPPPZ03}).
In the final Section \ref{holsec} we study the holonomy of the ambient metric of a $pp$-wave in relation to results in \cite{leistner05a}. We show that it is contained in the stabiliser of a totally null plane.

 \section{The Feffermann-Graham ambient metric}\label{ambient}
An important tool
 in order to construct invariants in conformal geometry is the
 so-called {\em Fefferman-Graham ambient metric} or {\em ambient
   space} (see \cite{fefferman/graham85} and
 \cite{fefferman-graham07}). Let $ (M,[g])$ be a  a smooth
 $n$-dimensional manifold $M$ with conformal structure $[g]$ of
 signature $(p,q) $ with the conformal frame bundle ${\mathcal
   P}^0$. It can also be characterised by a principle $\rr^+$-fibre
 bundle $\pi: {\mathcal Q}\rightarrow M$ defined as the ray sub-bundle in the bundle of metrics of signature $(p,q)$ given by metrics in the conformal class $c$. 
%In fact, $\cal Q$ is the  $\mathsf{det}^{w/n}$--reduction of $\cal P^0$ to $\cal Q$. 
The action of $\rr^+$ on ${\mathcal Q}$ shall be denoted by $\vf$:
\begin{eqnarray*}
\vf(t,g_x)&=& t^2 g_x.\end{eqnarray*} 
From \cite{fefferman-graham07} we adopt the following notation.
\begin{definition}\label{ambientmet}
Let $(M,[g])$  be a conformal structure of signature $(p,q)$ over an $n$-dimensional manifold $M$, and $\pi:{\mathcal Q}\rightarrow M$ the corresponding ray bundle.
 A semi-Riemannian manifold $(\wt{M},\wt{g})$ of signature $(p+1,q+1)$ is called {\em pre-ambient space} if 
\bnum
\item there is a free $\rr^+$-action $\widetilde{\vf}$ on $\wt{M}$, and
\item an embedding $\iota:{\mathcal Q}\rightarrow \wt{M}$ is $\rr^+$-equivariant.
%, i.e. the diagram
%\[ \begin{CD}
%\rr^+ \times \cal Q @>{\vf}>> \cal Q \\
%@V{id\times\iota}VV  @VV{\iota}V\\
%\rr^+\times \wt{M} @>{\wt{\vf}}>> \wt{M}
%\end{CD}\]
%commutes.
\item If $F$ is the fundamental vector field of $\wt{\vf}$, and ${\mathcal L}$ denotes the Lie derivative, then
${\mathcal L}_F\tg = 2 \tg$,
%\eeqa
i.e. the metric $\tg$ is homogeneous of degree $2$ w.r.t. the $\rr^+$-action.
\item Any $g_x\in {\mathcal Q}$ satisifies the equality
$
(\iota^*\tg )_{g_x} = g_x \left(d\pi(.),d\pi(.)\right)$
 in $\odot^2 T^*_{g_x}{\mathcal Q}$.
\enum
A pre-ambient space is called {\em ambient space} if its 
\emph{Ricci curvature vanishes}.
\end{definition}

Under the assumption that the conformal structure is given by a real analytic metric,
in odd dimensions a Ricci-flat ambient metric always exists and is also real analytic. 

In even dimensions $n\ge 4$, the existence of a Ricci-flat ambient
metric is obstructed by the nonvanishing 
of the
obstruction tensor ${\mathcal O}$,
\cite[pp. 22]{fefferman-graham07}. This is a symmetric trace-free and
divergence-free $(2,0)$-tensor, which is conformally invariant of
weight $(2-n)$, i.e. if $\hat{g}=\mathrm{e}^{2\vf} g\in [g]$, then 
$\hat{\mathcal O}=\mathrm{e}^{(2-n)\vf}{\mathcal O}$. It is given by
$$
{\mathcal O}=\D^{n/2-2}_g\left( \D_g \ro - \nabla^2 J\right) +\text{lower order terms,}
$$ 
where $\ro=\frac{1}{n-2}\left( Ric - \frac{scal}{2(n-1)}g\right)$ is the Schouten tensor, $J$ its trace, and $\D_g$ denotes the Laplacian of $g\in [g]$. 
For a conformal class in even dimension that is given by a real analytic metric with vanishing obstruction tensor, the ambient metric exists and is also real analytic.

Fixing a metric $g$ in the conformal class, in \cite{fefferman/graham85,fefferman-graham07} it is shown that an ambient space near
$M$ can be written as
$$\tem= (-\epsilon,\epsilon)\times M \times (1-\delta, 1+\delta) $$
with the ambient metric
$$
\tg = 2 t \der\rho \der t+ 2\rho \der t^2 + t^2g(\rho),
$$
in which $g(\rho)$ is a one-paramemter family of of metrics on $M$ with $g(0)=g$. This is referred to as $\tg$ being in {\em normal form}. 
As the ambient metric is analytic, one can write  the family $g(\rho)$ as a power series in $\rho$,
$$ 
\tg = 2 t \der\rho \der t+ 2\rho \der t^2 + t^2
\left( g + \rho g^\prime +\einhalb \rho^2 g^{\prime\prime} + \frac{1}{6}\rho^3g^{\prime\prime\prime}+ \ldots \right),
$$
with $g^\prime=\del_\rho g(0)$. 
 We summarise the results for the ambient metric in 
\begin{theorem}[\cite{fefferman/graham85,fefferman-graham07} and \cite{kichenassamy04}]
Let $(M,[g])$ be a real analytic manifold $M$ of dimension $n\ge 2$ equipped with a  conformal structure defined by a real analytic semi-Riemannian metric $g$. 
\bnum
\item If $n$ is odd, or if $n$ is even with $\mathcal O=0$, then there exists an ambient space $(\wt{M},\wt{g})$ with real analytic Ricci-flat metric $\wt{g}$.
\item If $n$ is odd
the ambient space 
 is unique modulo diffeomorphisms that restrict to the 
identity along ${\mathcal Q}\subset \wt{M}$ and commute with $\wt{\vf}$. If $n$ is even with $\mathcal O=0$, the ambient space is unique, modulo the same set of diffeomorphisms and modulo terms of order $\ge n/2$  in $\rho$, where $\rho $ is the coordinate in the normal form of the ambient metric.
\enum
 \label{fg}
\end{theorem}

The Ricci-flat condition then determines symmetric $(2,0)$-tensors $\mu_k$ such that
 \begin{eqnarray*}
\tg &=& 2 t \der\rho \der t+ 2\rho \der t^2 + t^2
\left( g +  \sum_{k=1}^\infty \rho^k\mu_k \right).
%\\
% &=& 2\left( t \der\rho + \rho \d t\right) \d t + t^2 g +  \rho t^2 \alpha  +(t\rho)^2\left(\beta + \rho\gamma + \sum_{k=4}^\infty \rho^k\mu_k \right)
%\left( \trace (\ro^2_p) -\frac{1}{(n-4)}B_p\right),
\end{eqnarray*}
In \cite{fefferman-graham07} the first $\mu_k$ are determined  explicitely:
\begin{equation}\label{abc}
\begin{array}{rcl}(\mu_1)_{ab}&=& 2\ro_{ab}
\\
(n-4)(\mu_2)_{ab}&=&
-B_{ab} +(n-4)\ro_a{}^c\ro_{bc}
\\
3(n-4)(n-6)(\mu_3)_{ab}&=&
\D_g B_{ab} - 2W_{cabd}B^{cd} -4(n-6)\ro_{c(a}B_{b)}{}^c-4\ro_c{}^c B_{ab}
\\
&&{ }
  +4(n-4)\ro^{cd}\nabla_d C_{(ab)c}
 -2(n-4)C^c{}_a{}^d C_{dbc}
 \\
&&{ }
  +(n-4)C_a{}^{cd}C_{bcd}
  +2(n-4)\nabla_d\ro^c{}_c C_{(ab)}{}^d
  \\
&&{ }
-2(n-4)W_{cabd}\ro^c{}_e\ro^{ed},

\end{array}
\end{equation}
where $W_{abcd}$ is the Weyl tensor, $\ro_{ab}$ is the Schouten tensor, $C_{abc}:=\nabla_c\ro_{ab}-\nabla_b\ro_{ac}$ is the Cotton tensor, and $B_{ab}= \nabla_c C_{ab}^{\ \ c}- \ro_{cd} W^{c\ \ d}_{\ ab}$ is the Bach tensor.

\section{$pp$-waves and their curvature}
\label{ppwaves}

A $pp$-wave is a Lorentzian manifold with a parallel  null vector field $K$, i.e. $K\not=0$ and $g(K,K)=0$, whose curvature tensor satisfies the 
trace condition
\begin{equation}\label{ppcurv}R_{ab}^{\ \ ef}R_{efcd}=0.\end{equation}
If we denote by $\kappa$ the one-form given by $\kappa:=K\hook g$ the curvature condition (\ref{ppcurv}) is equivalent to each of the following, in which $[ab]$ denotes the skew symmetrisation w.r.t. $i$ and $j$, \cite{schimming74}: 
 \begin{enumerate}
\item \label{eq1}
$ \kappa_{[a}R_{bc]de}=0$;
\item\label{eq2}
there is a symmetric  $(2,0)$-tensor $\varrho$ with $K\hook\varrho=0$, such that 
$  R_{abcd} = \kappa_{[a}\varrho_{b][c}\kappa_{d]}$;
\item\label{eq3}
there is a function $\vf$, such that
$R_{\ ab}^{e\ \ f}R_{ecdf} =\vf \kappa_a\kappa_b\kappa_c\kappa_d$.
\end{enumerate}
The Ricci tensor of a $pp$-wave is given by 
$Ric=\Phi\  \kappa\otimes\kappa$,
for a smooth function $\Phi$. In dimension $n=4$ this is even equivalent to the curvature condition (\ref{ppcurv}).

In \cite{leistner05a} we gave another equivalent definition, not using coordinates or traces, but identifying a $pp$-wave as a Lorentzian manifold with parallel null vector field $K$, whose curvature satisfies
\begin{equation}
\label{ppeinfach1}
{\mathrm Im} \Big( \cur (U,V)_{| K^\bot}\Big)= \rr\cdot K \mbox{ for all }U,V\in TM.
%\cur (U,V): K^\bot \longrightarrow \rr\cdot K \mbox{ for all }U,V\in TM.
\end{equation}
This equivalence allows for several generalisations \cite{leistner05c} and for an easy proof of another equivalence that is related to holonomy: An $n$-dimensional  Lorentzian manifold is a $pp$-wave if and only if its holonomy group is contained in the Abelian subgroup $\rr^{n-2}$ of the stabiliser in $\mathrm{SO}(1,n-1)$ of a null vector \cite{leistner01}.  

Locally, an $n$-dimensional  $pp$-wave admits coordinates $(x^1,\ldots , x^{n-2}, u,r)$ such that the metric is given by 
\be
g=\sum_{i=1}^{n-2}(\der x^i)^2~+~ 2\der u ~\big(\der r+h\der u\big),\label{pp3}
\ee
with $h$ being  a smooth real function  of the first $(n-1)$ coordinates, $h=h(x^i,u)$, \cite{schimming74}. In these coordinates the parallel null vector field $K$ is given by $\del_r$ and, up to symmetries, the only non-vanishing curvature terms of a $pp$-wave are
$$R(\del_i,\del_u,\del_j,\del_u)\ =\ \del_i\del_j h.$$
Here 
we use the obvious notation $\del_r:=\frac{\partial}{\partial r}$, $\del_u:=\frac{\partial}{\partial u}$ and $\del_i:=\frac{\partial}{\partial x^i}$, $i=1,\ldots , n-2$.
Hence, 
the function determining the Ricci-tensor is given by
$\Phi=-\D h$ with $\D h = \sum_{i=1}^{n-2}\del_i^2 h$, i.e.
\begin{equation}\label{ppric}
Ric = -\D h~ \der u^2.
\end{equation}
Hence, the Ricci-tensor is totally null, and the scalar curvature vanishes.
With this at hand, one can easily calculate the tensors related to the conformal geometry of a $pp$-wave. First, there is the Schouten-tensor
\begin{equation}\label{ppschouten}
\ro \ =\ \frac{1}{n-2}Ric\  =\  -\frac{\D h}{n-2}~ \der u^2.
\end{equation} 
Secondly, the Weyl tensor is given by
\begin{equation}\label{ppweyl}
W(\del_i,\del_u,\del_j,\del_u)\ =\  
\del_i\del_j h -\delta_{ij}\frac{\D h}{n-2},
\end{equation}
and for $n>3$ we obtain that $\del_i\del_j h =\delta_{ij}\frac{\D h}{n-2}$ as an equivalent condition on $h$ for $g$ being conformally flat.

Next, we calculate the Cotton tensor  $C$. As $\nabla\ro= - \frac{1}{n-2}\der (\D h)\otimes du^2$ one obtains that
\begin{equation}\label{ppcotton}C(\del_u,\del_i,\del_u)= - C(\del_u,\del_u,\del_i)=  \frac{\del_i \D h}{n-2}\end{equation}
are the only non-vanishing components of the Cotton tensor. Hence,
$\del_i \D h=0$ is
the condition on $h$ for $3$-dimensional conformally flat $pp$-waves.

Furthermore, we obtain the Bach tensor $B$,
 \begin{equation}\label{ppbach}
B= -\frac{\D^2h}{n-2}\  du^2.
\end{equation}
This enables us to calculate the next terms  in the ambient metric expansion in Eq.'s (\ref{abc}) beyond $\mu_1=2\ro = \frac{\D h}{n-2} du^2$, namely
\barr{rcccl}
\mu_2 &=&-\frac{1}{n-4} B &=&   \frac{\D^2h}{(n-2)(n-4)}  du^2,
\\
\mu_3&=& \frac{1}{2(n-4)(n-6)} \D B& =&
 \frac{\D^3h}{3(n-2)(n-2)(n-4)}  du^2.
 \earr
The very simple structure of $\mu_1$, $\mu_2$, and $\mu_3$ above, and
in particular the appearence of the consecutive powers of the
Laplacian, 
suggests that
this pattern may be also present in the next terms in the ambient metric
expansion. That this is really the case will be proven in the next section.

\section{The $pp$-wave ambient metric}\label{ppambient}
Looking at the very simple form of the $pp$-wave metric (\ref{pp3}) and 
the general formula for the ambient metrics (\ref{afg}), we 
\emph{guess} that the ambient metric for this $g$ is
\be\label{ppans}
\bar{g}=2\der (\rho t)\der t +t^2\Big(2\der u ~
\big(\der r+(h+H)\der u\big)~+~\sum_{i=1}^{n-2}(\der x^i)^2\Big),
\ee
where $H=H(\rho,x^i,u)$, and \be H(\rho,x^i,u)_{|\rho=0}=0.\label{hh}\ee

If we were able to find an analytic 
function $H$ satisfying (\ref{hh}) and for which the metric (\ref{ppans}) was \emph{Ricci flat} 
then, by the \emph{uniqueness} of the Fefferman-Graham Theorem \ref{fg}, we 
would conclude that $\bar{g}$ with this $H$ is the ambient metric 
for (\ref{pp3}). Thus to check our guess it is enough to calculate the
\emph{Ricci tensor} for (\ref{ppans}) and to check if its
\emph{vanishing} is possible for the function $H$ in the postulated
form (\ref{hh}).     

\begin{lemma}
The Ricci tensor of the metric (\ref{ppans}) is
$$Ric(\bar{g})=\Big( (2-n)H_\rho+2\rho H_{\rho\rho}-\triangle H-\triangle
h\Big)\der u^2.$$
Here $\triangle H=\sum_{i=1}^{n-2}\frac{\partial^2H}{\partial
  (x^i)^2}$, $H_\rho=\frac{\partial H}{\partial \rho}$, etc.
\proof{We start with a coframe 
\begin{equation}\label{ambcoframe}\theta^0=\der(\rho t),\ \theta^i=t\der x^i,\ \theta^{n-1}=t^2(\der
  r+(h+H)\der u),\ \theta^n=\der u,\ \theta^{n+1}=\der t,\end{equation}
in which the metric $\bar{g}$ reads:
$$\bar{g}=\bar{g}_{\mu\nu}\theta^\mu\theta^\nu=2\theta^0\theta^{n+1}+2
\theta^{n-1}\theta^n+\sum_{i=1}^{n-2}(\theta^i)^2,\quad\quad\mu,\nu=0,1,\dots,n+1.$$
It has the following differentials:
\barr{lcl}
\der\theta^0&=&0,\\
\der\theta^i&=&-t^{-1}\theta^i\dz\theta^{n+1},\quad\quad\quad\quad\forall
  i=1,\dots,n-2,\\
\der\theta^{n-1}&=&t H_\rho\theta^0\dz\theta^n+
t\sum\limits_{i=1}^{n-2}(h_i+H_i)\theta^i\dz\theta^n- 2t^{-1}\theta^{n-1}\dz\theta^{n+1}+\rho t
H_\rho \theta^n\dz\theta^{n+1},\\
\der\theta^n&=&0,\\
\der\theta^{n+1}&=&0.
\earr
In this coframe the Levi-Civita connection 1-forms, i.e. matrix-valued
1-forms satisfying 
$\der\theta^\mu+\Gamma^\mu_{~\nu}\dz\theta^\nu=0$,
$\Gamma_{\mu\nu}+\Gamma_{\nu\mu}=0$,
$\Gamma_{\mu\nu}=\bar{g}_{\mu\sigma}\Gamma^\sigma_{~\nu}$, are: 
\begin{equation}
\label{gammaij}\begin{array}{lcl}
\Gamma_{0n}&=&-t H_\rho\theta^n,\\
\Gamma_{in}&=&  -t(h_i+H_i)\theta^n,\\
\Gamma_{n-1~n} &=&t^{-1} \theta^{n+1}\\
\Gamma_{i\ n+1}&=&t^{-1}\theta^i,\\
\Gamma_{n-1~n+1} &=&t^{-1} \theta^{n}\\
\Gamma_{n\ n+1}&=& t^{-1}\theta^{n-1}-\rho t H_\rho\theta^n. 
\end{array}\end{equation}
Modulo the symmetry $\Gamma_{\mu\nu}=-\Gamma_{\nu\mu}$ all other
connection 1-forms are zero.  

The curvature 2-forms
$\Omega_{\mu\nu}=\der\Gamma_{\mu\nu}+\Gamma_{\mu\rho}\dz\Gamma^\rho_{~\nu}$,
have the following nonvanishing components:
\begin{eqnarray}
\Omega_{0n}&=&-H_{\rho\rho}\theta^0\dz\theta^n-\sum_{i=1}^{n-2}H_{i\rho}\theta^i\dz\theta^n- \rho
  H_{\rho\rho} \theta^n\dz\theta^{n+1},\nonumber\\
\Omega_{in}&=&
-H_{i\rho}\theta^0\dz\theta^n
-\sum_{k=1}^{n-2}(\delta_{ik}H_\rho+H_{ik}+h_{ik})\theta^k\dz\theta^n \label{rr}
 -\rho  H_{i\rho}\theta^n\dz\theta^{n+1},\\
\Omega_{nn+1}&=&
-\rho H_{\rho\rho}\theta^0\dz\theta^n
-\sum_{i=1}^{n-2}\rho H_{i\rho}\theta^i\dz\theta^n
  -\rho^2H_{\rho\rho} \theta^n\dz\theta^{n+1},\nonumber
\end{eqnarray}
together with the components that are implied by the symmetry
$\Omega_{\mu\nu}=-\Omega_{\nu\mu}$.

The Riemann tensor $R_{\mu\nu\rho\sigma}$, defined by
$\Omega_{\mu\nu}=\tfrac12
R_{\mu\nu\rho\sigma}\theta^\rho\dz\theta^\sigma$, can be read off from
the equations (\ref{rr}). Using it and the inverse of the metric
$g^{\mu\nu}$, $g_{\mu\rho}g^{\rho\nu}=
\delta_\mu^{~\nu}$, we calculate the Ricci tensor
$R_{\mu\nu}=g^{\rho\sigma}R_{\rho\mu\sigma\nu}$. It turns out that it  
has $R_{nn}=-2R_{0nnn+1}+\sum_{i=1}^{n-2}R_{inin}$
as its only nonvanishing component. Explicitly: 
$$R_{nn}=2\rho H_{\rho\rho}-(n-2)H_\rho-\triangle H-\triangle h.$$
This finishes the proof of the Lemma.}
\end{lemma}

The Lemma shows that the metric $\bar{g}$ is Ricci flat if and only if
the function $H$ satisfies the following PDE:
\be
(2-n)H_\rho+2\rho H_{\rho\rho}-\triangle H=\triangle h.\label{uh}\ee
For $\bar{g}$ to be the ambient metric for (\ref{pp3}) we in addition
require the initial condition (\ref{hh}). By looking for the solution
of the initial value problem (\ref{uh}), (\ref{hh}) in the form of a
power series
\be
H=\sum_{k=0}^\infty a_k\rho^k,\label{so}\ee
we immediately get $a_0=0$ from the initial condition (\ref{hh}). Then 
inserting (\ref{so}) in (\ref{uh}), we easily arrive at 
\begin{proposition}
If $n=2s+1$, $s\geq 1$, then the initial value problem (\ref{uh}), (\ref{hh}) has a \emph{unique}
power series solution. It is given by:
\be
H=\sum_{k=1}^\infty \frac{\triangle^kh}{k!\prod_{i=1}^k(2i-n)}\rho^k.\label{sol}\ee
If $n=2s$ the power series solution exists only if
$\triangle^{s}h=0$. If this is the case, the solution is also unique and
given by the power series (\ref{sol}), which truncates to a
\emph{polynomial} of order $(s-1)$ in the variable $\rho$. 
\end{proposition}
This proposition proves our Theorem \ref{theo1} of the introduction. 
Note that the solution we found is a solution to Equation 3.17 in \cite{fefferman-graham07} that was derived for the Taylor expansion of the ambient metric, here specified for a $pp$-wave.
In particular, for $n=2s$ the obstruction tensor of an $n$-dimensional $pp$-wave is given by
\[\mathcal O\ =\ \D^{s}h ~ du^2.\]
With this result at hand, every polynomial $h$ in the $x^i$'s of order lower than $2k$,  with 
coefficients being functions of $u$, gives an    example of a 
$pp$-wave for which the ambient metric truncates to a polynomial of order lower than 
$k$. This gives plenty of examples of explicit ambient metrics, also  in even dimensions. Moreover, choosing $h$ properly, one gets examples for which the 
conformal class does not contain an Einstein metric. This will be the aim of Section \ref{bachflat}. But first we address the issue of convergence of $H$ in odd dimensions.

\section{Convergence in three dimensions}

In odd dimensions the solution to the Ricci-flat equation, $H$  in (\ref{sol}),  may be given by an infinite series. Since  $H$ contains only natural powers of $\rho$, general arguments as in \cite{fefferman-graham07} ensure that $H$ converges for an analytic function $h$ and is analytic as well, \cite{callrobin}. 
Here we give a simple argument that proves convergence for $n=3$:
\begin{proposition}
Let $h$ be a function on $\bbC\times\bbR$ of variables $(z,u)$ which is an 
entire holomorphic function in $z=x+iy\in\bbC$, is continuous in
$u\in\bbR$, 
and is real for $z=x\in\bbR$. Then the series   
\be
H(x,u,\rho)=\sum_{k=1}^\infty \frac{(\triangle^kh)(x,u)}
{k!\prod_{i=1}^k(2i-3)}\rho^k\label{so3}\ee
converges uniformly on compact subsets of $\bbR^3$.

\proof{Let $R>1$ be a real number and let $C=\sup\{|h(z,u)|\}$ over all
values of $(z,u)$ such that $|z-x|\leq (R+2\epsilon)$, $|u|\leq\nu>0$,
and $|x|\leq\epsilon>0$. Then by the Cauchy-Schwarz inequality, the 
$k$th derivative of $h$ at every real point 
$(x,u)\in[-\epsilon,\epsilon]\times[-\nu,\nu]$ satisfies 
$|h^{(k)}(x,u)|\leq \frac{Ck!}{R^k}.$ This provides the following
estimate for the values of the powers of the Laplacian 
$\triangle^k h=\tfrac{d^{2k}h}{dz^{2k}}$:
\be
\forall(x,u)\in[-\epsilon,\epsilon]\times[-\nu,\nu]\quad{\mathrm we~ have}\quad|(\triangle^kh)(x,u)|\leq \frac{C(2k)!}{R^{2k}}.\label{so4}\ee  
Now we rewrite (\ref{so3}) to the equivalent form 
%\marginpar{I dropped the $(-1)^k$ here!}
$$H=\rho\triangle
h- \sum_{k=1}^\infty\frac{\triangle^{k+1}h}{(k+1)!\cdot 1\cdot
  3\cdot\dots\cdot (2k-1)}\rho^{k+1}.$$
To show that $H$ converges it is enough to show the convergence of the
power series above. This can be done by using the estimate (\ref{so4}):
\begin{eqnarray*}
&&|\sum_{k=1}^\infty\frac{\triangle^{k+1}h}{(k+1)!\cdot 1\cdot
  3\cdot\dots\cdot (2k-1)}\rho^{k+1}|\leq C\sum_{k=1}^\infty\frac{(2k+2)!}{(k+1)!\cdot 1\cdot
  3\cdot\dots\cdot (2k-1)}\left(\frac{|\rho|}{R^2}\right)^{k+1}\\
&&=C\sum_{k=1}^\infty\frac{(2\cdot 4\cdot\dots\cdot 2k)\cdot (2k+1)(2k+2)}{(k+1)!}\left(\frac{|\rho|}{R^2}\right)^{k+1}=C\sum_{k=1}^\infty b_k\left(\frac{|\rho|}{R^2}\right)^{k+1}.\end{eqnarray*}
Since 
$$\frac{|b_{k+1}|}{|b_k|}=\frac{2(k+1)(2k+3)(2k+4)}{(k+2)(2k+1)(2k+2)}\longrightarrow
2\quad{\mathrm as}\quad k\to\infty,$$
then this series converges for $|\rho|\leq\frac{R^2}{2}.$ This
finishes the proof.}
\end{proposition}

\section{Bach flat metrics that are not conformally Einstein}
\label{bachflat}
With Eq. (\ref{ppbach}) it is obvious how to obtain Bach-flat $pp$-waves. It is more difficult to find those that are not conformally Einstein. 
In this section we want to give examples of $4$-dimensional $pp$-waves
that are both, Bach flat, and not conformal to Einstein. But first we have to review some necessary condition of being conformal to Einstein given in  \cite{gover/nurowski04} for any dimension.

From the formulae for the transformation of the Schouten tensor under conformal changes of the metric one obtains that a metric is conformal to an Einstein metric if and only if there exists a scaling function $\ups$ such that 
\begin{equation}
\label{ppconfeinstein}
\ro - \nabla \der \ups +(\der\ups)^2 \text{ is pure trace.}
\end{equation}
In the following we write $Y$ for the gradient of $\ups$. 
In \cite[Proposition 2.1]{gover/nurowski04} the following 
necessary
conditions for the metric to be conformal to Einstein were derived from Eq. (\ref{ppconfeinstein}):
\begin{eqnarray}
\label{212}
C+W(Y,.,.,.)&=&0\\
 B+ (n-4) W(Y,.,.,Y) &=&0\label{213}.
\end{eqnarray}
Note that the first condition is satisfied for a gradient $Y$ if and only if 
the metric is conformally equivalent to a metric with vanishing  Cotton tensor, i.e. if it is {\em conformally Cotton-flat}.
We further mention that
the property of being conformally Cotton-flat is also neccessary for the
metric to be conformally Einstein \cite{gover/nurowski04}.

For a $pp$-wave conditions (\ref{212}) and (\ref{213}) are equivalent to the following:
\bs\label{prop21}
If the $pp$-wave (\ref{pp3}) is conformally Einstein but not conformally flat and $n>3$, then there is a vector field  $Y$ on $M$, whose components 
 $Y^i:=\der x^i(Y)$, $i=1,\ldots , n-2$, and $Y^{n-1}:=\der u(Y)$ satisfy the equations
\begin{eqnarray}
\del_i \D h - Y^i  \D h +(n-2)
\sum_{k=1}^{n-2}
 Y^k  \del_k\del_i h&=&0\label{pp212}
\\
\D^2 h  - (n-4)  \D h \
\sum_{k=1}^{n-2} \left(Y^k \right)^2
+ (n-2)(n-4)
\sum_{k,l=1}^{n-2} 
Y^k Y^l   \del_k\del_l h &=& 0\label{pp213}
\end{eqnarray}
for $i=1,\ldots , n-2$, and 
\begin{equation}
Y^{n-1}=0.\label{duy}
\end{equation}
\es
%%\begin{remark}
%%\begin{align*}
%%\theta^i&=dx^i,& \theta^{n-1}&=du,&\theta^n&=dr+hdu\\
%%\intertext{with the dual basis}
%%E_i&=\del_i,& E_{n-1}&=\del_u-h\del_r&E_n&=\del_r.
%%\end{align*}
%%I.e. $Y=\sum_{k=1}^n Y^\mu E_\mu$ with $Y^\mu=\theta^\mu(Y)$ and in particular $Y^i=dx^i(Y)$ and $Y^{n-1}=du(Y)$ 
%%}\end{remark}

\bprf
Writing $Y =  Y^k\del_k +Y^{n-1} \del_u+ \d r(Y) \del_r$, Eq. (\ref{212}) and the formulae in Section \ref{ppwaves} 
give
\begin{eqnarray*}
0&=&  Y^{n-1} W(\del_u,\del_i,\del_u,\del_j)\\
0&=&
\frac{\del_i \D h}{n-2}+ Y^k\left( \del_k\del_i h -\delta_{ki} \frac{ \D h}{n-2}\right).
\end{eqnarray*}
These, when  $n>3$, imply both, $Y^{n-1}=0$ and Eq. (\ref{pp212}).
Equation (\ref{213})  gives that
\begin{eqnarray*}
0&=&
% B(\del_z,\del_z) +(n-4) b^kb^l W(\del_k,\del_z,\del_z,\del_l)\\&=&
-\frac{\D^2 h}{n-2} -(n-4)Y^kY^l \left(\del_k\del_l h- \delta_{kl}\frac{\D h }{n-2}\right)
\end{eqnarray*}
which implies Eq. (\ref{pp213}).
%Now, $Y$ is a gradient of a smooth function of $\vf$, and with the help of the Witt basis (\ref{ppwitt}) can be written as 
%$$Y
%=
%\left( \del_z \vf -2h\del_x\vf\right)\del_x +\sum_{k=1}^{n-2}\del_k\vf \del_k+ \del_x \vf \del_x.$$
%Hence, $\del_x \vf=c=0$ and $\del_k\vf =b_k$ satisfying Eqs. (\ref{pp212}) and (\ref{pp213}). 
\eprf
Writing $Y$ as the gradient of $\Upsilon$, 
$$Y
=
 \sum_{k=1}^{n-2}\del_k\ups \del_k+ \del_r \ups \del_u  +\left( \del_u \ups-h\del_r\ups\right)\del_r.$$
 the proposition implies that $\der u(Y) =\del_r \ups =0$. Hence,
$$
\del_r\left(\der r(Y)\right) = \del_r \left( \del_u \ups -h\del_r\ups\right)=0,
$$
and we obtain
\begin{corollary}\label{prop21cor}
Let $g$ be a  $pp$-wave that is conformally Einstein but not conformally flat in dimension $n>3$, and let $Y$ be the gradient of the scaling function $\ups$ satisfying Eq. (\ref{ppconfeinstein}). Then the function $Y^n=\der r (Y)$ does not depend on the $r$-variable.
\end{corollary}
%{\tt \edz{We may skip this remark.}Note that in fact $dr(Y)=Y^n:=\theta^n(Y)=dr(Y)-h du(Y)$ because $du(Y)$ =0.}

%\bprf
%%As $Y$ is the gradient of a function,  the expression $g(\nabla Y, .)$ is equal to the Hessian  of that function and thus a {\em symmetric} $(2,0)$-tensor. 
%%Since $\der u (Y)=0$, 
%%this implies that $\der r(Y)$ does not depend on $r$:
%%$$ 
%%0=g (\nabla_{\del_u}Y,\del_r)= g (\nabla_{\del_r}Y,\del_u)=\del_r\left( \der r(Y)\right).
%%$$
%%The first equality holds because $Y\in \del_r^\bot$ which is a parallel distribution.
%%\eprf  
%%\begin{remark}
%Writing $Y$ as the gradient of $\Upsilon$, 
%$$Y
%=
% \sum_{k=1}^{n-2}\del_k\ups \del_k+ \del_r \ups \del_u  +\left( \del_u \ups-h\del_r\ups\right)\del_r.$$
% the proposition impplies that $\der u(Y) =\del_r \ups =0$ and hence
%$$
%\del_r\left(\der r(Y)\right) = \del_r \left( \del_u \ups -h\del_r\ups\right)=0.
%$$
%\eprf
%\end{remark}

\bbsp \label{ex3dim}
For $n=3$ a third order polynomial $h$ in $x$ with coefficients being functions of $u$ defines a $pp$-wave with non-vanishing Cotton tensor. Hence, it is not conformally flat and therefore not conformally Einstein.
\ebsp
\bbsp\label{exmoredim}
Set $M=\rr^n$ and  $h= x_1^4+ \ldots +x^4_{n-2}$. Then, 
$\del_i\del_j h \not= \delta_{ij}\frac{\D h}{n-2}$ on open sets in $M$ and hence, $g$ is not conformally flat. On the other hand,  Eq. (\ref{pp213}) can never be satisfied in $0\in M$, because here all second order derivatives of $h$ vanish, but  $\D^2h = 24(n-2)$. Thus, the $pp$-wave defined by $h= x_1^4+ \ldots +x^4_{n-2}$ is not conformally Einstein.
\ebsp

Now we turn to  dimension $n=2s=4$. Here the formula (\ref{sol}) makes sense only if
$\triangle^2h=0$. In such case the formula truncates to
$H=\frac12\rho\triangle h$. 
Thus it is clear that for the
4-dimensional $pp$-waves the Fefferman-Graham obstruction is
\emph{precisely} $\triangle^2 h$, which is a multiple of the Bach tensor, and does not involve any lower order 
terms in the derivatives of the metric functions. In order to write down all
such metrics, it is convenient to pass to the \emph{complex notation}
by introducing coordinates $z=\frac{x^1+ix^2}{\sqrt{2}}$,
$\bar{z}=\frac{x^1-ix^2}{\sqrt{2}}$. In this notation the \emph{most
general} 4-dimensional $pp$-wave metric \emph{satisfying} $\triangle^2h=0$ is
given by
$$g_4=2\der u\Big(\der
r+\left(\bar{z}\alpha+z\bar{\alpha}+\beta+\bar{\beta}\right)\der
u\Big)+2\der z\der\bar{z}.$$
Here $\alpha=\alpha(z,u)$, $\beta=\beta(z,u)$ are \emph{holomorphic} functions of
$z$. This metric is \emph{Bach-flat}, and in \emph{some} cases, 
such as when 
$a_z+\bar{\alpha}_{\bar{z}}={\mathrm const}$, is conformal to an Einstein
metric. 
Its ambient metric  is given by 
$$\tilde{g}_4=2\der(\rho t)\der t+t^2\Big(2\der u[\der
r+\left(\bar{z}\alpha+z\bar{\alpha}+\beta+\bar{\beta}-\rho(a_z+\bar{\alpha}_{\bar{z}})\right)\der
u]+2\der z\der\bar{z}\Big),$$
and by construction is \emph{Ricci flat}. We get
\bs
A $4$-dimensional $pp$-wave $g_4$ is Bach flat if and only if 
$$g_4=2\der u\Big(\der
r+\left(\bar{z}\alpha+z\bar{\alpha}+\beta+\bar{\beta}\right)\der
u\Big)+2\der z\der\bar{z},$$
with $\alpha=\alpha(z,u)$, $\beta=\beta(z,u)$ functions of a complex variable $z$ and a real variable $u$ which are holomorphic in
$z$. 
\es
In general,  this Bach-flat metric is \emph{not}
conformally Einstein:
\begin{theorem}
A $4$-dimensional Bach-flat $pp$-wave 
\begin{equation}
\label{confcott}
g_4=2\der u\Big(\der
r+\left(\bar{z}\alpha+z\bar{\alpha}\right)\der
u\Big)+2\der z\der\bar{z}
\end{equation}  with  $\beta\equiv 0$ is conformally equivalent to a metric with vanishing Cotton tensor. Moreover, the following three properties are equivalent:
\bnum
\item  $\del_z^2\alpha\equiv 0$,
\item $g_4$ is conformally flat,
\item $g_4$ is conformally Einstein.
\enum
In particular, any such metric with $\del_z^2\alpha\not\equiv 0$
is not conformally Einstein.
\end{theorem}

\begin{proof}
First we trivially get 
that in the complex coordinates $(z,\bar{z})$ we
have: $\D h=2\left( \del_z \alpha
+\del_{\bar{z}}\bar{\alpha}\right)$. 
Next, using   
$$
\del_1 = \frac{1}{\sqrt{2}}\left(\del_z +\del_{\bar{z}}\right),\ \ 
\del_2 = \frac{\mathrm{i}}{\sqrt{2}}\left(\del_z -\del_{\bar{z}}\right),
$$
in the formula (\ref{ppweyl}) we see that the Weyl tensor vanishes if and only if $\del_z^2\alpha=0$. This proves the equivalence of (1) and (2).

For the remaining statements we try to find a vector field $Y$ that solves the necessary condition (\ref{212})
for $g$ to be conformally
Einstein. We use this equation in the form (\ref{pp212}), as in
Proposition \ref{prop21}. Recall that in this proposition we proved
that such a vector does not have a $\del_u$-component. Thus we look
for $Y$ of the form 
$$
Y =   F\del_z+ \overline{F}\del_{\bar{z}} + f \del_r
$$
 where $F=F(z,\bar{z},r,u)$ is a complex and $f=f(z,\bar{z},r,u)$ is
a real function.
Eq. (\ref{pp212}) gives
\begin{eqnarray}
\label{1}0&=&
\del^2_z \alpha \left( 1+ \bar{z} F\right) + \del^2_{\bar{z}}\bar{\alpha}\left( 1+ z\overline{F}\right)
\\
\label{2}0&=&
\del^2_z \alpha \left( 1+ \bar{z}F\right) - \del^2_{\bar{z}}\bar{\alpha}\left( 1+ z\overline{F}\right)
\end{eqnarray}
which immediately implies 
$$\del^2_z \alpha \left( 1+ \bar{z} F\right)=0.$$
Assuming that $g_4$ is not conformally flat, i.e. $\del_z^2\alpha\not\equiv 0$ we get
$$ F(z)=- 1/\bar{z}.$$
Thus we found that the vector $Y$ solves (\ref{212}) if and
only if 
$Y = -\frac{1}{\bar{z}}\del_z - \frac{1}{z}\del_{\bar{z}}+f\del_r$.
Now, $g_4$ is conformally Cotton-flat if we find $f$ such that  this $Y$ is a gradient.  Setting
$$Y^\flat=g_4(Y,.)=  -\frac{1}{{z}}\der z -\frac{1}{\bar{z}}\der \bar{z}+f\der u,$$
we see that $Y$ is locally a gradient, i.e. $\d Y^\flat=0$, if and only if $f$ is a
function of variable $u$ alone. Every $f=f(u)$ gives a solution to the
conformally Cotton equation.
%\edz{BTW this observation is due to  Brinkmann, I guess}

To prove that (3) implies (2), assume that 
 $g_4$ is not conformally flat but conformally Einstein.
 Then we plug in the vector $Y^\flat$ we have obtained as a solution of Eq. (\ref{pp212}), and its corresponding 
 $$
%\label{nabups}
\nabla Y^\flat 
=
\der f \otimes \der u - \left( \frac{\alpha + z\del_{\bar{z}}\bar{\alpha}}{\bar{z}} + \frac{\bar{\alpha} + \bar{z}\del_{{z}}{\alpha}}{{z}}\right) \der u^2 
+\frac{1}{{z}^2}\der z^2+\frac{1}{\bar{z}^2}\der \bar{z}^2
$$
into
$$
\ro -\nabla Y^\flat +(Y^\flat)^2.
$$
According to Equation (\ref{ppconfeinstein}) this must be a pure
trace, if the metric $g_4$ is conformally Einstein. But this can not
happen since $\ro -\nabla Y^\flat +(Y^\flat)^2$ has a
nowhere vanishing  $\der z\der \bar{z}$-term given by
$\frac{2}{z\bar{z}}\der z\der \bar{z}$, and an 
identically vanishing $\der r \der u$-term. Thus $\ro -\nabla
Y^\flat +(Y^\flat)^2$ is never proportional to $g_4$,
which in turn, can not be conformally Einstein.
\end{proof}
In the light of discussions in
 \cite{gover/nurowski04}, the metrics (\ref{confcott}) provide interesting examples because, apart from being Bach-flat, they are conformally Cotton-flat, but {\em not} conformally Einstein even though the necessary conditions (\ref{212}) and (\ref{213}) are both satisfied for a gradient.

We strongly believe that a similar argument works in any dimension, even though one might  not be able to describe the functions with $\D^{s}h=0$. But under certain assumptions it might be possible to deduce a contradiction between Eq.'s (\ref{pp212}) -- (\ref{pp213}) and the fact that  the function $\der r(Y)$ is independent of the $r$-coordinate as it occurs for $n=4$.

We want to conclude this section by returning to the result of Brinkmann in \cite{brinkmann25} mentioned in the introduction. If a $4$-dimensional $pp$-wave is Einstein, and hence Ricci-flat, the function $h$ is given by $\alpha+\overline{\alpha}$ for a holomorphic function $\alpha$.  Again, this metric is conformally flat if and only if $\del_z^2\alpha=0$. If it is not conformally flat but conformally Einstein, then the vector field $Y$ is  null and a multiple of $\del_r$, namely $Y=f\del_r$ with a function $f=f(u)$ that depends on the variable $u$ only. As $\ro=0$, Equation \ref{ppconfeinstein} then is equivalent to $\dot{f}=f^2$. Hence, any such function yields a conformal rescaling of a Ricci-flat $pp$-wave to another Einstein  metric that is in fact Ricci-flat. The new metric may  be isometric to the the original one but in general this is not the case (see also \cite{ehlers-kundt62}).
 
\section{The critical $Q$-curvature of a $pp$-wave}
\label{qcurv}
For a semi-Riemannian manifold of $(M,g)$ even dimension $n=2s$, in  \cite{branson93}
 T. Branson introduced  a series $\{Q_{2k}\}_{k=1\ldots s}$ of scalar invariants constructed from the curvature tensor involving $2k$ derivatives of the metric\footnote{Regarding this section, we would like to thank Andreas Juhl for explaining to us some facts about $Q$-curvature.}. 
As such, for a $pp$-wave all $Q_{2k}$ are zero. This follows form the general fact, that all scalar invariants constructed from the Riemannian curvature tensor of a $pp$-waves vanish (for a proof in arbitrary dimension see \cite{CMPPPZ03}).
However,  as an application of Theorem \ref{theo1},
in this section we will use the $pp$-wave ambient metric in order to show that the {\em critical $Q$-curvature} $Q_n$ of a $pp$-wave vanishes.
The so-called {\em subcritical} $Q$-curvatures $Q_2, \ldots , Q_{n-2}$ are defined by the inhomogeneous part of the GJMS-operators $P_{2k}$, namely
\[
P^g_{2k}(1)=(s-k) Q_{2k}.
\]
The GJMS-operators $P_{2k}$  introduced in \cite{gjms} are conformally covariant operators. We will not give a definition of the {\em critical} $Q$-curvature $Q_n$ here (please refer to \cite{fefferman/hirachi03}, for example).
Instead we will explain a formula for the critical  $Q$-curvature given in \cite{graham-juhl07} that expresses it in terms of the volume of the Poincar\'e metric. 

Let $(M,[g])$ be a smooth  manifold of even dimension $n=2s$ with conformal class $[g]$. To this manifold one can assign a Poincar\'e metric $g_+$. $g_+$ is a metric on  $M_+=M\times (0,a)$ given by
\[g_+ = \frac{1}{x^2}\left( \der x^2 + g_x\right)\]
where $g_x$ is a $1$-parameter family of metrics with the same signature as $g$ and with initial condition $g_0=g$ such that $g_+$ is asymptotically Einstein, which means that 
$Ric(g_+)+ng_+$ vanishes up to terms of order $(n-2)$ in $x$. The Poincar\'e-metric is unique up to addition of terms of the form $x^n S_x$ where $S_x$ is a $1$-parameter family of symmetric $(2,0)$-tensors such that $S_0$ is trace-free.
(for details see \cite{fefferman/graham85,fefferman-graham07}). For a Poincar\'e metric one can show, see  \cite{graham99} for details, that  $\sqrt{\det(g_x)/\det(g)}$ has the Taylor expansion
\be\label{volume}
\sqrt{\frac{\det(g_x)}{\det(g)}}= 1+ v^{(2)}x^2 + v^{(4)}x^4+ \ldots + v^{(n-2)}x^{n-2}+ v^{(n)}x^n + \ldots \ee
defining smooth functions $v^{(2k)}$. Then in \cite{graham-juhl07} it is shown that the critical  $Q$-curvature $Q_n$  of $(M,[g])$ is  given as
\be\label{qform}
2nc_{\frac{n}{2}}Q_n= nv^{(n)} + \sum_{k=1}^{s-1} (n-2k) \mathcal A_{2k}^* v^{(n-2k)}.\ee
Here $\mathcal A_{2k}$ are the  linear differential operators that appear in the expansion of a harmonic function for a Poincar\'e-metric, the star denotes the formal adjoint, and $c_{\frac{n}{2}}$ is a constant.
%%\[Q= 2 \left( J^2 -\|\ro\|^2\right) +\D_g J,\]
%%and for $n=6$
%%\begin{eqnarray*}
%%Q & = &8\ro^{ab}B_{ab} +16 \ro^{ab}\ro_a{}^c\ro_{cb}-24J\|\ro \|^2+8J^3\\
%%&&{ } +\Delta_g^2J+4\Delta (J^2) +8\nabla_b(\ro^{ab}\nabla_a J)-4\Delta_g(\|\ro \|^2), 
%%\end{eqnarray*}
%%where $J$ is the trace of the Schouten-tensor, i.e. a multiple of the scalar curvature, and $\D_g$ denotes the Laplacian w.r.t. the metric $g$. As the scalar curvature of a $pp$-wave vanishes and because of the degeneracy of $\ro$ and $B$, we obtain

Furthermore, one has to recall how the Poincar\'e-metric can be obtained by the ambient metric. Assume that $$\tg=2\der (\rho t)\der t+t^2g(\rho)$$ is a pre-ambient metric for $[g]$ that is Ricci-flat up to terms of order $s$ and higher. Such a metric always exists and is unique up to terms of order $n/2$ in $\rho$.  Now, on
$$M_+=\{ (\rho,p,t) \in \widetilde{M} \mid p\in M, t^2\rho=-1\},$$
the Poincar\'e-metric is given by
$$ g_+= \frac{1}{x^2}\left( \der x^2 +\frac{1}{2} g(x^2)\right).$$	
Note that if the pre-ambient metric is Ricci-flat, then the Poincar\'e-metric obtained in this way is Einstein. We can use the ambient metric of a $pp$-wave  to prove
\begin{theorem}
The critical $Q$-curvature of an even-dimensional  $pp$-wave vanishes.
\end{theorem}
\begin{proof}Let $(M,g)$ be a $pp$-wave of even dimension $n=2s$. In Section \ref{ppambient} we have also shown that its pre-ambient metric  that is Ricci-flat up to terms of order $n/2$ is given by
formula 
(\ref{ppans}) with $H$ as in (\ref{sol}). 
Using the coframe in (\ref{ambcoframe}) we can write down the volume form $\w(\rho)$ of the $\rho$-dependend family of $pp$-waves
$$g(\rho) = 2\der u ~
\big(\der r+(h+H)\der u\big)~+~\sum_{i=1}^{n-2}(\der x^i)^2,
$$ namely
$$ \w(\rho)= \der x^1\wedge \ldots \wedge \der x^{n-2}\wedge \left( 
\der r+(h+H)\der u\right)\wedge \der u
=\w(0).$$
For the family $g_x=\frac{1}{2} g(x^2)$ defining the Poincar\'e metric this implies 
 that $\det(g_x)=\det (g_0)$. Hence, all the $v^{(2k)}$ in (\ref{volume}) are zero and so is the critical  $Q$-curvature by the result of \cite{graham-juhl07} given in formulae (\ref{qform}).
\end{proof}
Recall that for a $pp$-wave $(M,g)$ the vanishing of the scalar curvature implies that the Laplacian $\D_g$ is conformally covariant.
Calculations using formulae  in \cite{juhl08} show that  the first GJMS-operators $P_2$, $P_4$ and $P_6$  are equal to the corresponding powers of the Laplacian $\D_g$, $\D_g^2$ and $\D_g^3$. We conjecture that for $pp$-waves this is also the case for the higher $P_{2k}$.

\section{Conformal and ambient holonomy}
\label{holsec}

%\edz{This is new.}
We conclude with a brief remark about the holonomy of the ambient metric and the holonomy of the normal conformal Cartan connection, also called the {\em conformal holonomy}, of a $pp$-wave. Holonomy groups describe the reduction of generic structures down to more special structures, in the semi-Riemannian, the conformal, and in other geometric settings. For a conformal manifold of signature $(r,s)$ the conformal holonomy is contained in $\mathrm{SO(r+1,s+1)}$. If it is a proper subgroup, then the conformal structure is reduced to a more special structure. Examples are Lorentzian Fefferman spaces, for an overview see \cite{baum07}, where the conformal holonomy reduces to the special unitary group, or conformal structures in signature $(2,3)$ with non-compact $G_2$ as structure group, \cite{nurowski04,nurowski07}.

In \cite{leistner05a} it is proven that the conformal holonomy of an $n$-dimensional Lorentzian conformal class that is given by a metric with parallel null line and totally null Ricci tensor is contained in the stabiliser 
%$\mathrm{Iso}_{\mathrm{SO(2,n)}}(\mathcal N) $ 
in $\mathrm{SO(2,n)}$ of a totally null plane $\mathcal N$. Of course, $pp$-waves are special examples of such metrics and hence, their conformal holonomy reduces to this stabiliser. But we get the same result also for the holonomy of the ambient metric of  a $pp$-wave.
\bs
The metric $\ol{g}$ defined in Eq. (\ref{ppans}) admits a holonomy invariant distribution of totally null planes ${\mathcal N}$ spanned by $\del_r $ and $\del_\rho$. In particular, all curvature operators $\bar{R}(V,W)$,  $V,W\in T\bar{M}$, leave invariant the fibres of ${\mathcal N}$ and of ${\mathcal N}^\bot$, which is spanned by $\del_r$, $\del_\rho$, and $\del_i$.
\es

\bprf
The easiest way to see this is to consider the dual frame to the co-frame in (\ref {ambcoframe}) given by
$$
E_0=\frac{1}{t}\del_\rho,\  E_i=\frac{1}{t}\del_i,\  E_{n-1}= \frac{1}{t^2}\del_r,\  E_n=\del_u -
(h+H)\del_r,\ E^{n+1}=\del_t -\frac{\rho}{t}\del_\rho.
$$
Using the relation 
$\bar{g}(\bar{\nabla} E_\mu,E_\nu)= \Gamma_{\mu \nu}$ one can read off from the formulae for the connection $1$-forms in (\ref{gammaij}) that
$$ \mathcal N =\mathrm{span}(E_0,E_{n-1})=\left( \mathrm{span}(E_0,E_i,E_{n-1})\right)^\bot
$$ is invariant under the Levi-Civita connection.
\eprf
\begin{corollary}
Let $G$ be the holonomy group of the ambient metric of a $pp$-wave  in odd dimension or in dimension $2s$ with $\D^s h=0$. Then $G$
 is contained in the stabiliser in $\mathrm{SO(2,n)}$ of a totally null plane in $\rr^{2,n}$.
 \end{corollary}
In general, it is possible to show that the conformal holonomy is always contained in the ambient holonomy \cite{g2inprep}. For a conformal class with an Einstein-metric or a Ricci-flat metric both holonomy groups are the same \cite{leistner05a,leitner05}.  For a $pp$-wave, not necessarily conformal Einstein,  
we have just seen that both are contained in the isotropy group of a totally null  plane. Hence,  it is very likely that the conformal holonomy is actually {\em equal} to the ambient holonomy. But to give a proof of this is beyond the scope of this paper.

%\bibliography{GEOBIB,SPINBIB,HOLBIB,ALGBIB,thomas,CONF}

\def\cprime{$'$}

\end{document}